\documentclass[12pt]{article}
\usepackage{version}
\usepackage{amsmath}
\usepackage{amsfonts}
\usepackage{amssymb}
\usepackage{amsthm}
\usepackage{txfonts}
\usepackage[all]{xy}
\usepackage{eucal}

\def\OO{{\mathcal O}}

\def\g{\mathfrak g}

\def\Spec{\mathop{\rm Spec}\nolimits}

\newtheorem{theorem}{Theorem}
\newtheorem{proposition}[theorem]{Proposition}
\newtheorem{lemma}[theorem]{Lemma}
\newtheorem{definition}[theorem]{Definition}
\newtheorem{corollary}[theorem]{Corollary}

\newtheorem{remark}[theorem]{Remark}

\begin{document}

\baselineskip=16pt

\centerline{\Large\bf Rationality of the instability parabolic and related results}

\bigskip

\centerline{\bf Sudarshan Gurjar, Vikram Mehta}

\bigskip

\begin{abstract}
In this paper we study the extension of structure group of principal bundles with a reductive algebraic group
as structure group on smooth projective varieties defined over algebraically closed field of positive
characteristic. Our main result is to show that given a representation $\rho$ of a reductive algebraic group
$G$, there exists an
integer $t$ such that any semistable $G$-bundle whose first $t$ frobenius pullbacks are semistable induces a
semistable vector bundle on extension of structure group via $\rho$. Moreover we quantify the number of such
frobenius pullbacks required.

\end{abstract}

\section{Introduction}
Let $X$ be a smooth projective variety defined over an algebraically closed field $k$. Fix a
very ample line bundle $H$. Let $G$ be a reductive algebraic group defined over $k$. All representations
considered in this paper are rational finite-dimensional representations. Recall that a $G$-bundle is
semistable with respect to the polarisation $H$ if for any reduction of structure group to a parabolic
subgroup
$P$ and any dominant character of P, the induced line bundle on $X$ has non-positive degree.\\
$~$ Now let $\rho: G \rightarrow Gl(V)$ be a rational representation of $G$ sending the connected component
of the centre of $G$ to that of $Gl(V)$. If the characteristic $k$ is zero and $E$ is a semistable $G$-bundle
on $X$ then the induced
$Gl(V)$-bundle is also semistable. From this it follows easily that if the characteristic of the field is
\textquotedblleft sufficiently large \textquotedblright, then again a semistable $G$-bundle induces a
semistable $Gl(V)$-bundle. This is
quantified in [IMP]
where in it is shown that if char $k >$ ht $(\rho)$, then a semistable $G$-bundle induce a semistable
$Gl(V)$-bundle on extension of structure group. In positive characteristic however, it is in not in
general true that a semistable $G$ bundle will induce a semistable $Gl(V)$-bundle. 
  A principal $G$-bundle on $X$ is said to be strongly semistable if all its frobenius
pullbacks are
semistable. In char 0, the frobenius map is just identity and hence the notion of semistability and strong
semistability coincide. In Ramanan-Ramanathan [RR], it is shown that a strongly semistable $G$-bundle induces
a strongly semistable $Gl(V)$-bundle. This result is sharpened in the paper of Coiai-Holla [CH] where the
authors show that given a representation $\rho$ as before, there exists a non-negative integer $t$ such that
if
$E$ is any $G$-bundle on $X$ which along with its first $t$ frobenius pullbacks is semistable, then the
induced $Gl(V)$-bundle is again semistable. This fact is crutial in their proof of boundedness of semistable
$G$-bundles with fixed Chern classes. In this paper we give bounds for this $t$, in terms
of certain numerical data attached to $G$ and $\rho$. The main ingedient
of the proof is the use of the instability
parabolic (also known sometimes as the Kempf's parabolic) associated to points of the representing space (see
[Section 3] for
definition). The basic idea is as follows:
     Let $E$ be a principal $G$-bundle on $X$. Let $k(X)$ denote the function field of $X$. Let $E(G)$ be the
group scheme associated to $E$ (see section 2 for definition). Let $E_{Gl(V)}$ denote the induced $Gl(V)$
bundle. Let $E(G)_\circ$ denote the generic fiber of $E(G)$. It is a group scheme defined over the function
field of $X$. Let $P$ be any maximal parabolic in $Gl(V)$. Let
$E(Gl(V)/P)$ be the associated $Gl(V)/P$ fiber-space. Again let $E(Gl(V)/P)_\circ$ denote the generic fiber of
$E(Gl(V)/P)$. Then $E(G)_\circ$ acts on $E(Gl(V)/P)_\circ$ which is linearized by a suitable very ample line
bundle. If $E_{Gl(V)}$ admits a reduction of structure group to this maximal parabolic $P$, then we get a
section (canonically) $\sigma$ of $E(Gl(V)/P)$. Restricting to the generic fiber gives a $k(X)$-valued point
$\sigma_\circ$ 
of $E(Gl(V)/P)_\circ$. In [RR], it is shown that if either $\sigma_\circ$ is a semistable point for action
described above or its instability parabolic (see [Section 3] for definition), which is in general defined
over $\bar{k(X)}$, is actually defined
over $k(X)$, then this section (or equivalently this reduction) does not contradict semistability. In char 0,
using  uniqueness of the instability parabolic and Galois descent, this proves the semistability of the
induced bundles.
 In [CH] it is shown that that there exists a non-negative integer $t$ such that for all possible reductions
to
all the maximal parabolics the instability parabolics of points corresponding to these reductions is
actually defined over $k(X)^{p^{-t}}$. This can be shown to imply that if $E$ is a semistable principal
$G$-bundle with first $t$
frobenius pullbacks semistable, the induced $Gl(V)$-bundle is also semistable. 
  The main aim of this paper is to give bounds for this $t$ in terms of certain numerical data attached to 
$G$ and $\rho$. \\

\section{Basic definitions and preliminary notions}{\label{Basic definition}}

 In this section we set up some notations and recall some of basic definitions and facts which will be used
later. \\
$X$ will always denote a smooth projective variety over a field $k$. Let $G$
be a reductive algebraic group defined over $k$ and let $\g$ denote its lie algebra. Fix a maximal torus $T
\subset G$ and a Borel
$B$ containg $T$. Let $X_*(T)$ denote the set of 1-parameter subgroups of $T$ and let $X^*(T)$ be
the
character group of $T$. There exists a nondegenerate pairing, denoted $(\cdot , \cdot)$ : $X_{*}(T) \times
X^*(T)
\rightarrow \mathbb{Z}$.  Let $\Phi \subset X^*(T)$ be the set of roots of $G$. Let $\Phi^+$ denote the set of
positive roots corresponding to
the choice of $B$ in $G$ and $\Delta $ = \{$\alpha_1, \cdots ,\alpha_n \}$  a set of simple roots of $G$.
Corresponding to this choice of simple roots, there exists a set of elements $\omega_i \in
X_*(T) \otimes {\mathbb Q}$ known as the fundamental weights with the property that $\langle
\omega_i,\alpha_j \rangle=\delta_{ij}$.\\
For any root $\alpha$, there exists an isomorphism of $x_\alpha$ of $G_a$ with a closed subgroup $X_\alpha$
of $G$ with the property that $t \cdot x_\alpha(a) \cdot t^{-1} = x_\alpha(\alpha(t)a)$. $X_\alpha$ is known
as the root group associated to $\alpha$.\\
 By a parabolic in $G$, we mean a closed subgroup of $G$ containing $B$. There exists a natural bijection
of the set of subsets of $\Delta$ with the set of parabolic subgroups of $G$ containing $B$ under which for
a subset $I \subseteq \Delta$, we assign the parabolic $P_I$ to be the closed subgroup of $G$ generated by
$B$ and $X_{^+_- \alpha}$ for all roots $\alpha \in \Delta \backslash I$.
 Let $W=N(T)/T$ be the Weyl group. Fix a $W$-invariant inner product $\langle$ , $\rangle$ on  $X_*(T)
\otimes 
\mathbb {Q}$. Using this inner product we can define norm of any
1-PS $\lambda(t) \in T$ as $\mid \mid \lambda(t) \mid \mid  = \langle \lambda, \lambda \rangle$.
For a arbitrary 1-PS in $G$ we can conjugate it into the fixed maximal torus and then define its norm.
\\ We begin by recalling the definitions of semistability of vector and principal bundles with respect to
the fixed polarisation $H$. 
\begin{definition}
  For a vector bundle $E$ on $X$, define its slope to be the rational number:  
 \[  \mu(E)= deg(E)/rk(E).
 \]
 $~$A vector bundle $E$ on $X$ is said to be {\bf $\mu$-semistable} (w.r.t. the polarization $H$) if for
any proper subbundle $F \subset E$,
we have  the inequality $\mu(F) \leq \mu(E) $, where $\mu$ denotes the slope of the bundles. 
  
  For any vector bundle $E$, there exists a canonical filtration of $E$ by $\OO_X$-coherent subsheaves known
as its Harder-Narasimhan filtration (denoted HN$(E)$). 
\[
0 \subset E_1 \subset \cdots \subset E_l = E
\] 
with the property that successive quotients $E_i/E_{i-1}$ are $\mu$-semistable and $\mu (E_i/E_{i-1}) > \mu
(E_{i+1}/E_i)$ for all $1 \leq i \leq l$.

Define $\mu_{max}(E) = \mu (E_1)$ and by $\mu_{min}(E)=\mu(E_l/E_{l-1})$
    
    The quantity $\mu_{max}(E)-\mu_{min}(E)$ known as the {\bf instability degree} is a measure of the
instability
of the vector bundle. 
  
\end{definition}

  \begin{definition}
A principal $G$-bundle $E$ over $X$ is said to be semistable if for any reduction of structure group to a
parabolic $P$ of $G$ and any dominant character on $P$, the induced line bundle has degree $\leq$ 0.\\ 
$~$ Equivalently, a principal bundle $E$ on $X$ is said to be semistable is for any reduction of structure
group to a parabolic $P$ of $G$, the pullback of the relative tangent bundle of $E(G/P)$ over $X$, via the
section $\sigma: X \rightarrow E(G/P)$ corresponding to this reduction is a vector bundle on $X$ of degree
$\geq$ 0. 
  \end{definition}
     
 \begin{definition}\label{H-N filtration}
   Let $E$ be a principal $G$ bundle on $X$. Let $E_P$ be a reduction of structure group af $E$ to a
parabolic $P \subset G$. The reduction is said to be {\bf canonical} (or the {\bf Behrand reduction
}) if the following conditions are satisfied:

\indent 1) deg $E_P(P) > 0$\\

\indent 2) For any parabolic subgroup scheme $Q \subset E(G)$, deg $Q \leq$ deg $E_P(P)$.\\

\indent 3) For any subgroup scheme $Q \supset E_P(P)$, deg $Q <$ deg $P$.\\

\indent 4)The unipotent radical bundle $E_P(P)/R_u(P)$ is semistable.\\

With these conditions our definition of canonical reduction coincides with that of Behrend.
$P$ is known as the {\bf Behrend's parabolic}. The
degree of $E_P(P)$ is denoted by $deg_{HN}(E)$.
 \end{definition}
     
 $~~$ The canonical reduction can be shown to be equivalent to the following: For any nontrivial character on
$P$ which is a
non-negative                     
combination of simple roots with respect to the choice of $B$, the induced line bundle on $X$ obtained by
extension of structure group has non-negative degree.


\section{Frobenius morphism}{\label{Frobenius morphism}}
 Let $X$ be a scheme over a algebraically closed field of char $p>0$. The $p$-th power map $\OO_X \rightarrow
\OO_X$ given by $f 
\rightarrow f^p$ gives rise to a morphism of schemes $F_X: X \rightarrow X$ called the absolute frobenius. If
$k$ is a perfect field, this morphism is an isomorphism (although not a $k$-morphism in general). Let $F^m$
denote the iterated frobenius map. If $E$ is a $G$-bundle on $X$ we an take its pullback $F^{m^{*}}(E)$ which
will be a $F^{m^*}(G)$ bundle. We call this the $m$-th frobenius pullback. By twisting Spec $k$ by the
frobenius map (which will be an isomorphism), we can define a $k$-structure
on $F^{m^*}(X)$, $F^{m^*}(G)$ as well as $F^{m^*}(E)$. The $G$ bundle $F^{m^*}(E)$ on $X$ is the same as the
one obtained by extension of structure group under the homomorphism $G \rightarrow G$ given by the $m$-th
frobenius map.\\
 Clearly if the frobenius pullback of a $G$-bundle is semistable with respect to the pulled back polarization,
then so is the original bundle. A semistable $G$ bundle may not however pullback to a semistable $G$-bundle.
A $G$-bundle $E$ is said to be {\bf strongly semistable} if all its frobenius pullbacks are also semistable.

\section{The instability parabolic}{\label{The instability parabolic}}        
      
$~$ In this section we discuss the role of the instability parabolic which plays an important role in studying
extension of structure groups in positive characteristic. We first begin by recalling some elementary notions
and
facts from Geometric Invariant Theory. \\
 \indent Let $K$ be an algebraically closed field. Let $G$ be a reductive algebraic group defined over $K$.  
Let $\rho: G \rightarrow Gl(V)$ be a representation of $G$ defined over $K$. A vector $v
\in V$ is said to
be semistable for the $G$-action if $0 \notin \bar{Gv}$. Equivalently there exists a G-invariant $\phi \in
S^n(V)$ for some $n>0$ such that $\phi(v) \neq 0$.\\
For a 1-PS $\lambda(t)$ of $G$ we get a decomposition of $V = \oplus~
V_i$, where $V_i=\{v \in V \mid \lambda(t)(v)=t^i(v) \} $.\\
Define $m(v,\lambda)$ = min \{i $\mid v$ has a nonzero component in $V_i$\}.\\
Define slope of the 1-PS $\lambda(t)$ by
\[
\nu(\lambda,v)= m(v,\lambda) / \mid \mid \lambda \mid \mid
\]

Note that for any vector $v \in V(\bar K)$ and any 1-PS $\lambda$, we have $\nu(\lambda,v)=\nu(g \lambda
g^{-1},gv)$

\begin{lemma}(See [RR])
 There exists a constant $C$ such that for all $v \in V$ and all 1-PS $\lambda$, $\nu(\lambda,v) \leq C$.
 \end{lemma} 
 
\indent For a non-semistable vector $v \in V$ define its {\bf instability 1-PS} (denoted $\lambda_v$) to be
one for
which
$\nu(v,\lambda)$ attains the maximum value among all the 1-PS of $G$. Intuitively, this is the 1-PS in
$G$
which takes the vector $v$ to $0$ fastest after proper scaling.

\indent For a 1-PS $\lambda$ define a parabolic $P( \lambda)$ whose valued points consist of elements $g \in
G$
such that $\underset{t \to 0}{\lim}~ \lambda(t)g\lambda(t)^{-1}$ exists. This is known as the {\bf instability
parabolic} associated to $\lambda$. If $\lambda$ is an instability 1-PS of $v$, then $P(\lambda)$ will also
be known as the instability parabolic of $v$, denoted $P(v)$.

  Now if $G$ acts on a projective variety $M$ defined over $K$ which is linearized by some very ample line
bundle $\mathcal L$, then we
get a $G$-equivariant embedding $i: M \hookrightarrow {\mathbb {P}}(H^\circ(M,{\mathcal L}))={\mathbb
{P}}(V)$. We then
say that a point $m \in M$ is semistable for the $G$-action if the corresponding point in $V$ is semistable.

 We recall some basic facts concerning instability 1-PS (See [RR] )

Suppose $G$ acts on a projective variety $M$ as above. Let $m \in M$ be a nonsemistable point for the action
of $G$. \\

 (a) The function which sends every 1-PS $\lambda$ of $G$ to $\nu(\lambda,m)$ attains its maximum on the set
of all 1-PS subgroups of G. Following [RR], we denote this value by $B$.\\

(b) There exists a parabolic subgroup $P(m)$ of $G$, called the instability parabolic associated to the point
$m$, such that for any instability 1-PS $\lambda$ associated to $m$, we have $P(m)=P(\lambda)$.\\

(c) The instability parabolic $P$ is generated by $T$ together with the root groups $U_\alpha$ correponding to
roots $\alpha$ for which $\alpha(\lambda) \geq 0$.\\

(d) A maximal torus $T$ in $G$ contains a instability 1-PS $\lambda$ for $m$ if and only if $T \subset
P(\lambda)$.
Such a 1-PS is neccessarily unique.\\

(e) For a non-semistable $m \in M$, if $\lambda(t)$ is an instability 1-PS of $m$, then $g\lambda(t)g^{-1}$
is the instability 1-PS of $gm$ and $\nu(\lambda,m)= \nu(g\lambda g^{-1},gm)$\\

(f) For a 1-PS $\lambda$ of $G$ and any element $g \in P(\lambda)$ we have $\nu(m,\lambda)=\nu(gm,\lambda)$.\\

(g) For any $g \in G$, we have $P(gm)=gP(m)g^{-1}$.\\

(h) If $m \in M$ is an unstable point for the action of $G$ having an instability 1-PS defined over an
extension field $[L:K]$, then the instability parabolic $P(m)$ is also defined over $L$. \\

\indent Now let $K$ be an arbitrary field (not neccessarily algebraically closed). Let $K_s$ denote its
seperable closure. Let $G$ be a reductive algebraic group defined over $K$. Let $T$ be a fixed
maximal torus of $G$ (which will always be split over $K_s$, in fact over a finite extension of K). Let $M$
be a
projective variety defined over $K$ on which $G$ acts, linearized by a very ample line bundle $\mathcal L$
giving a
$G$-equivariant embedding $i:M \hookrightarrow {\mathbb{P}}(V)$. Fix a inner product on $X_*(T \otimes K_s)$
to be one which is invariant under the action of the Weyl group as well as the Galois group Gal$(K_s \mid K)$
(See [Kempf]).\\
\indent A point $m \in M$ is said to be semistable if it semistable after base change to its algebraic
closure,
i.e thought of as an element in $V (\bar{K})$.

Let $m \in M$ be a $K$-rational point of $M$. Let $P(m)$ be the instability parabolic of $m$ defined
over
$\bar{K}$. By invariance of the inner product under the Galois action and uniqueness of $P(m)$ we see that
if 
$P(m)$ is defined over $K_s$, then it is already defined over $K$. [See RR].\\
  
   {\bf Rationality of the instability parabolic and its consequences}\\
   
  Let $X$, $G$ and $\mathcal L$ be as before. Suppoe $\rho: G \rightarrow Gl(V)$ be a representation of $G$
which
takes the connected component of the centre of $G$ to the centre of $Gl(V)$. Let $P$ be a maximal parabolic of
$Gl(V)$. Choose the very ample generator $\mathcal L$ of $Gl(V)/P$. This is a linearized very ample line
bundle giving an embedding of $Gl(V)/P$ inside a projective space $\mathbb {P}(W)$.\\
\indent  Now let $\pi:E \rightarrow X$ a principal $G$ -bundle on $X$. Let $E(G)$ be the associated group
scheme over
$X$. Let $E(Gl(V)/P)$ be the associated fiber space. Let $T_\pi$ denote the relative tangent bundle on
$E(Gl(V)/P)$. Let $E(\mathcal L)$ be the associated line bundle on $E(Gl(V)/P$ corresponding the line bundle
$\mathcal L$ on $Sl(V)/P$. The group scheme $E(G)$ acts on $E(G/P)$. Let
$E(G)_\circ$ be the generic fiber of $E(G)$. It is a group scheme defined over the function field of $X$.
$E(G)_\circ$ acts on $E(Gl(V)/P)_\circ$ which is linearized by $E(\mathcal L)_\circ$. Let suppose $\sigma$ be
a
reduction of the induced $Gl(V)$-bundle to $P$. Then corresponding to this reduction we get a section of
(called $\sigma$ again) of $E(Gl(V)/P)$ over $X$. Let $\sigma_\circ$ be the associated
$k(X)$-valued point of $E(Gl(V)/P_\circ)$. Suppose $\sigma_\circ$ is a non-semistable point for the action
of $E(G)_\circ$ on $E(Gl(V)/P)$. Let $P(\sigma_\circ)$ denote the instability parabolic
associated
to the point $\sigma_\circ$. We call $P(\sigma_\circ)$ the instability parabolic corresponding to this
reduction. Let $T_\sigma$ denote the pullback of $T_\pi$ via the
section $\sigma$.

\begin{proposition} (See [RR, Proposition 3.10, (1)])
 Let $\sigma_\circ$ be a semistable point for the action of $E(G)_\circ$ on $E(G/P)_\circ$ with respect to
the polarization $E(\mathcal L)_\circ$. Then $T_\sigma$ has degree $\geq 0$.\\ 
$~$ In other words this reduction of structure group does not contradict semistability of $E(Gl(V)$.
 
\end{proposition}

\begin{proposition}{\label{HC}} (See [RR])
 Let $E$ be a semistable $G$-bundle. Suppose for every reduction to a parabolic $P$ in $Gl(V)$, the
instability parabolic associated to this reduction is rational (defined over $k(X)$), then the induced $Gl(V)$
bundle is semistable.
\end{proposition}

\begin{proposition}{\label{field of definition}}   (See [HC], Proposition 4.5)
 Let $G$ be a reductive algebraic group defined over an arbitrary field $K$ (not neccessary algebraically
closed) acting on a projective variety $M$ defined over $K$. Then there exists an integer $t$
such that given any $K$-valued point
$m \in M$ which is not
semistable its instability parabolic $P(m)$ is defined over
$K^{p^{-t}}$. 
\end{proposition}


{\bf The method of Holla-Coiai}

  In this section we briefly explain the method of Holla-Coiai for proving the existence of the integer $t$
in proposition \ref{field of definition}.
We will be brief and sketchy in this exposition.  Let $G$ and $M$ be as in above proposition \ref{field of
definition}. Let $\mathcal L$ be a linearized very ample line bundle on $M$ giving a $G$-equivariant embedding
$i: M \hookrightarrow {\mathbb {P}}(H^\circ(M,{\mathcal L}))={\mathbb
{P}}(V)$  \\
For an affine algebra $A$ over $K$, we define its radical index to be the smallest integer $n$, such that
$f^n= 0$ for all $f \in$ Rad$({\bar A}) \overset{\text{by~defn}}{=}$ Rad $(A \otimes_{K} {\bar K})$.
Now let $m \in M$ be a $K$-rational point of $M$ which is not semistable for the $G$-action.\\
\indent  Recall that the the action of $G$ is said to be strongly seperable at a point $m \in M$ if the
isotropy
subgroup scheme at every ${\bar K}$-valued point in the closure of $O(m)$ is reduced, where $O(m)$ denotes the
orbit of $m$. Let $P(m)$ be the instability parabolic of $m$. There exists $g \in G$ such that the parabolic
$P = gP(m)g^{-1}$ is defined over $K_s$.  By uniqueness of the instability parabolic and
Galois descent, it is already defined over $K$. Let
$x_m=gm$. Then the instability parabolic of $x_m$ is $P$. Since $P(x_m)$ is defined over $K$, it contains a
maximal torus over $K$ (which is split over $K_s$). Hence there is a unique instability 1-PS of $x_m$
contained in this maximal torus which is defined over $K_s$ and hence by uniqueness defined over $K$. \\
\indent Consider the decomposition of $V = \oplus V_i$ into simultaneous eigenspaces for the action of
$\lambda$, where $V_i=\{v \in V \mid \lambda(t)(v)=t^i(v)\}$.
Let $j=m(x_m,\lambda)$ and $V^j=\oplus V_i, i\geq j$.
 Define the $K_s$-scheme $M(P)_{x_m}$ to be the scheme theoretic intersection of the $K_s$-subscheme
${\mathbb P}(V^j)$
and $O(m)$ of ${\mathbb P}(V)$. 
  The following proposition summarizes the basic properties of the scheme $M(P)_{x_m}$.
  
\begin{proposition}{\label{strongly seperable}}
 The $\bar K$-valued points of $M(P)_{x_m}$ are precisely those points in the $K$-scheme $O(m)$ for which the
instability parabolic is $P(x_m)$. Also, when the action of $G$ on $m$ is strongly seperable, then
$M(P)_{x_m}$ is absolutely reduced. 
\end{proposition}

Suppose one can find a $K_s$ rational point $m'$ in $M(P)_{x_m}$, then by proposition { \ref{strongly
seperable}}, its
instability parabolic being $P(x_m)$ is hence defined over $K_s$. Since $m$ and $m'$ are both $K_s$ rational
points,
they are translates of each other by a $G(K_s)$-valued point $g$ and hence their instability parabolic are
conjugates by $g$. This will prove that the instability parabolic for $m$ is defined over $K_s$ and hence by
uniqueness and Galois descent it is defined over $K$.\\
  Thus the problem of showing the existence of the integer $t$ in proposition {\ref{field of definition}}
boils down to finding a finite purely inseperable extension $L$ of $K_s$ (independent of the
point $m$) over which the scheme $M(P)_{x_m}$ will have a $L$-valued point. This bound is obtained
using the following lemma's:

\begin{lemma}
 Let $f: Y\rightarrow X$ be a morphism of finite-type scheme over $\bar K$. Then there exists an integer $n$ 
 such that the radical index of the schematic fiber of $x$ is less than or equal to $n$ for all closed points
$x \in X$. 
\end{lemma}

\begin{lemma}
 Let $A$ be an affine $K_s$-algebra with radical index $\leq p^n$. Then $A$ admits a $K_s^{p^{-n}}$-rational
point. 
\end{lemma}

\section{Bounds for the field of definition of the instability parabolic and its
consequences}{\label{Bounds for the
instability parabolic}}
  In this section we give explicit bounds for the field of definition of the instability parabolic associated
to non-semistable points for the action of a reductive algebraic group $G$ acting on a vector space
$V$ defined over an
arbitrary field $K$. We do this by giving explicit bounds for the field of definition for the instability 1-PS
associated to these points. We first do this $G=Sl(2)$, where we can
 get much better bounds than for a general $G$, then for the tensor power
representation of $Sl(n)$, then for an arbitrary representation of $Sl(n)$ and then for an
arbitrary
representation of any arbitrary reductive algebraic group $G$.\\ 
$~$ We now begin with giving bounds for the field of definition of the instability parabolic for various
$Sl(2)$-modules. 

\begin{lemma}{\label{Sl(2)}}
 Let $K$ be any field (not neccessarily algebraically closed) if char $p>0$. Let $G=Sl(2,K)$. Let
$\rho:Sl(2,K) \rightarrow
 S^N(V)$ be the standard symmetric power representation. Let $N=N_0 + N_1p + N_2p^2 + \cdots + N_tp^t$ be the
p-adic expansion of $N$. Then for any non-semistable $K$-rational vector $v \in S^N(V)$, the instability
parabolic $P(v)$ of
$v$ is defined
over $K^{1/p^t}$. 
\end{lemma}

\begin{proof}
 By uniquenes of instability parabolic and Galois descent explained before, we can assume that $K$ is
seperably closed. Let
$X,Y$ denote the basis for $V$ over $K$. Thus $S^N(V)$ can be identified with the vector space of all  
degree $N$ homogeneous
polynomials in $X$ and $Y$. Let $f=\underset {i+j=N}{\sum} a_{ij}X^iY^j, a_{ij} \in K$ be an unstable vector
in $S^n(V)$ for the action of
$G$. \\

Claim 1: $f$ has a zero of multiplicity greater than $N/2$ on $\mathbb{P}^1_{\bar K}$.\\
\indent Proof of claim: Let $\lambda(t)$ be the
instability 1-PS of $f$ defined over $\bar K$. Every 1-PS of $Sl(2)$ is conjugate over $\bar K$ to the 1-PS
\[
\mu(t)= \begin{pmatrix}
         t & 0 \\
         0 & t^{-1}
        \end{pmatrix}
\]
 Choose $g \in Sl(2, \bar K)$ such that $g\lambda(t)g^{-1}$ is of the form $\mu(t)$.
 Then $\mu(t)$ is the instability 1-PS of $g \cdot f$. Let suppose $g \cdot f$
have the form:
\[
 g\cdot f=X^Tg
\]
for some nonnegative integer $T$ and some polynomial $g \in S^N(V)$ which is not divisible by $X$. Since
$\mu(t)$
drives $g \cdot f$ to $0$, $T$ neccessarily satisfies $N/2 < T \leq N$. i.e. $f$ has a zero of multiplicity
greater than $N/2$ on $\mathbb{P}^1_{\bar K}$ and hence a unique such zero.
 \\

$~$Now, by using the fact that $K$ is seperably closed, by a suitable change of basis made over $K$, we can
assume that $f$ can be factorized in the form:
\[
 f= F_1 \cdot F_2 \cdots F_r
\]
for some $0 \leq r \leq N$, with deg $F_1 \geq$ deg $F_2 \geq \cdots \geq$ deg $F_r$
and each $F_i$ of the form $(X^{p^{t_i}}- \alpha_i Y^{p^{t_i}})$ for some non-negative integer $t_i$.\\
Note that ${t_1 \geq t_2 \geq ... \geq t_r \geq 0}$.  Factorizing $f$ into
product of linear polynomials over the field $K^{1/p^{t_1}}$, we get :
\[
 f= (X - \alpha_1^{1/p^{t_1}}Y)^{p^{t_1}} \cdots (X - \alpha_r^{1/p^{t_r}}Y)^ {p^{t_r}}
\]
Note that by Claim 1, $p^{t_1} = T$. By once again making a change of basis over the field $K^{1/p^{t_1}}$,
sending 
\[
 (X - \alpha_1^{1/p^{t_1}}Y) \rightarrow X'
\]
\[
Y \rightarrow Y' 
\]
 
and calling the resulting polynomial $f'$ (which is a translate of $f$ by an element in
$Sl(2,K^{1/p^{t_1}})$), we
see that $f'$ has the form 
\[
 f'= X'^{p^{t_1}}(X'-\beta_1 Y')^{p^{t_2}} \cdots (X'-\beta_r Y')^{p^{t_r}}
\]
with all the $\beta_i$'s distinct. Note that $\beta_1,...,\beta_r$ belong to $K^{1/p^{t_1}}$.
Since $f$ has a unique root of multiplicity $> N/2$, we see that the factor occuring in the above
factorization with
the highest power is neccessarily unique. i.e. $t_1$ is unique.\\

$~$ Claim 2: The 1-PS $\mu(t)$ is an instability 1-PS of $f'$.\\
Proof of claim 2 : The proof of the claim is quite obvious. We only sketch it briefly. Note that $\nu(f',\mu)$
= $t_1/(\mid\mid \mu \mid\mid)$. Suppose there exists another 1-PS $\mu'(t)$ such that $\nu(f',\mu') >
\nu(f',\mu)$. Since all 1-PS's of $G$ are conjugates over ${\bar K}$, there exists an element $h \in G(\bar
{K})$ which conjugates $\mu$ into $\mu'$. Then $\mu(t)$ will be the instability 1-PS of $hf'$. It is easy to
see that the
highest power of $X'$ occuring in $f'$ is greater than or equal to the highest power of $X'$ occuring in
$hf'$. Hence we see
that $m(f',\mu) \geq m(hf',\mu) = m(f',\mu')$. Since $\mu$ and $\mu'$ are conjugates over ${\bar K}$, we see
that
this implies that $\nu(f',\mu) \geq \nu(f',\mu')$. This proves that $\mu$ is an instability 1-PS of $f'$ and
hence completes the proof of Claim 2. 

Now since $f$ and $f'$ are translates of each other by an element in $K^{1/p^{t_1}}$ and an instability 1-PS
of $f'$ is defined over $K$, we see that an instability 1-PS and hence the instability parabolic of $f$ is
defined over
$K^{1/p^{t_1}}$. 

\end{proof}

\begin{corollary}
 Let $\rho : G \rightarrow S^N(V)$ be the representation as in lemma {\ref{Sl(2)}}. If $N > p$, the
instability
parabolic of any non-semistable vector in $S^N(V)$ is rational. 
\end{corollary}

\begin{proof}
Obvious.
\end{proof}

In general, for an arbitrary representation of $Sl(V)$, the method does not seem to work. This is because it
is in general impossible to determine all the non-semistable points in the representing space. Hence we have
adopt a more indirect way of bounding the field of definition of the instability 1-PS which does not use the
knowledge of all the non-semistable vectors. We begin with a lemma which will be a crutial step in the
bounding of the field of definition of the instability 1-PS :

\begin{lemma}{\label{Noether normalization}}
 Let $K$ be an infinite field. Let $A=K[Y_1,...Y_n]/(f_1,...,f_r)$ be a finitely generated $K$-algebra. Let
$g \in K[Y_1,...Y_n]$. Let suppose deg $f_i=d_i$. Let $d= \prod
d_i$. Let suppose $X$= Spec $A$ thought of as a closed subscheme of $\mathbb {A}^n_K$ has a $\bar K$-valued
point at which $g$ is non-vanishing (thought of as a regular function on $X$). Then there exists an
extension field $L$ of $K$ with deg $[L:K] \leq d$ such that $X$ has a $L$-valued point at which $g$ is
non-vanishing. 
\end{lemma}

\begin{proof}
 Let $V(g)\subset X$ be the closed subscheme of $X$ defined by the intersection of the vanishing locus of
$g$ with $X$. Let $X'=X \backslash V(g)$ be an open affine subscheme of $X$. Now by hypothesis $X$
has a $\bar K$-valued point. By restricting to a irreducible component of Spec $A$ containing the $\bar K$
valued point, we can assume that $X$ is irreducible. Let dim $X=m$. By a linear change of coordinates, we can
perform a Noether normalisation such there exists $m$
elements $t_1,...,t_m$ in $A$ such that $A$ is integral over $B=K[t_1,...,t_m]$ and the induced map $f$: Spec
$A \rightarrow$ Spec $B$ on affine schemes corresponding to the inclusion of $B$ in
$A$ has degree atmost $d$. Let $p\in B$
be a $K$ valued point of $B$ which is not in the image of $V(g)$. This is possible to choose since $f$ is a
finite map.
By
going-up lemma, there exists a point $q\in X'$ lying over $p$. Let the residue field extension $[K(q):K(p)]$
be $s$. Then $s\leq$ deg $f \leq d$. Taking $L$ to be $K(q)$, we get the lemma. 
\end{proof}

\begin{lemma}{\label{quantification}}
  Let $V$ be a vector space of dimension $n$ defined over a field $K$ of char $p>0$. Let $G=SL(V)$. Let $K$
be an arbitrary field of char $p>0$. Let $\rho: Sl(V) \rightarrow Sl(V^{\otimes m})$ be the tensor power
representation of $SL(n)$. Then for any non-semistable $K$-rational point $v \in
V^{\otimes m}$, the instability parabolic $P(v)$ is defined over an extension field of $[L:K]$ of degree $\leq
mn^m$.
Equivalently if $t$ is such that $p^t > mn^m$, then the instability parabolic for unstable $K$-rational point 
is defined over $K^{1/{p^t}}$.
\end{lemma}

\begin{proof}
 Let $X_1,...,X_n$ be a basis of $V$ over $K$. By uniqueness of instability parabolic and Galois descent, we
may assume that all the objects are defined over the seperable closure $K_s$ of $K$. Hence without loss of
generality we may assume $K=K_s$. Let $R=K\langle X_1,..,X_n \rangle$ denote the non-commutative
polynomial ring in the variables $X_1,...,X_n$. Let $R^m$ denote the vector subspace of $R$ consisting of
non-commutative monomials in $X_1,...,X_n$ of degree $m$. Let $w_1,...,w_M$ denote an ordered basis of
$R^m$ consisting of non-commutative monomials of degree m (words). Note that $M=n^m$. Then $V^{\otimes m}$ can
be identified
with $R^m$, the identification compatible with the action of $Sl(V)$. For any extension field $[L:K]$
we will think of elements $g \in G(L)$ as $n \times n$ matrices $g_{ij}$ with coefficients in $L$.
\\Consider the commutative polynomial ring $B=K[G_{ij}]$. Any $g=g_{ij} \in G(L)$ can thus be thought of as a
$L$-valued point of Spec $B$. Let $v=\sum a_iw_i$ be any element in $R^{\otimes m}$ . We define the elementary
polynomials associated to $v$ as follows:\\
\indent Denote by $K[G_{ij}]
\langle
X_1,...,X_n \rangle$ the noncommutative
ring in the variables $X_i$, with coefficients in the
commutative polynomial ring $K[G_{ij}]$.  Consider the set mapping $\theta: K\langle X_1,...,X_n \rangle
\rightarrow K[G_{ij}] \langle
X_1,...,X_n \rangle$ defined as follows:\\
\indent $\theta$ sends a variable $X_i$ in $K\langle X_1,...,X_n \rangle$ to $\sum G_{ij} \cdot X_j$ and
extends the
action in the obvious way to $K\langle X_1,...,X_n \rangle$. The ordered set of coefficients of the various
noncommutative monomials in the $X_i$'s that occur in $\theta v$ (which are polynomials in the
commutative ring $K[G_{ij}]$) will be called the elementary polynomials corresponding to $v$, denoted $EP_v$
(some of which
may be the zero polynomial for a given $v$). More precisely, if $\theta (v) =  \sum f_{i_v} w_i$, with
$f_{i_v} \in K[G_{ij}]$, then the set ordered set ${f_{1_v},f_{2_v},...,f_{M_v}}$ will be defined to be the
elementary polynomials associated to $v$. Just for the sake of clarity we explain this definition (of
elementary polynomials) by taking a simple example.\\
$~$ In the two-variable case, consider the action of $Sl(2,K)$ on $V^{\otimes 2}$
as above. If
$\{X_1^2, X_1X_2, X_2X_1 , X_2^2\}$ denote the ordered basis for $V^{\otimes 2}$, then the elementary
polynomials  associated to the vector $v=X_{1}^2 + X_1 \cdot X_2$  will be computed as follows:
Consider the image of $v$ under $\theta$:
   \[  X_1 \rightarrow G_{11}X_1 + G_{12}  X_2 \]
   \[  X_2 \rightarrow G_{21}X_1 + G_{22} X_2 \]
 Hence the image of $v=X_{1}^2 + X_1  X_2 $ will be: \\
 $(G_{11}X_1 + G_{12}X_2)^2 +
(G_{11}X_1 + G_{12} X_2)  (G_{21}X_1 + G_{22}X_2) \\
= (G_{11}^2X_1^2 +
G_{11}G_{12}X_1 X_2 + G_{12}G_{11}X_2 X_1 + G_{12}^2 X_2^2) + (G_{11}G_{21}X_1^2 + G_{11}G_{22}X_1X_2 +
G_{12}G_{21}X_2X_1 + G_{12}G_{22}X_2^2)\\
= (G_{11}^2 + G_{11}G_{21})X_{11}^2 + (G_{11}G_{12} + G_{11}G_{22})X_1X_2 + (G_{12}G_{11} +
G_{12}G_{21})X_2X_1 + (G_{12}^2 + G_{12}G_{22})X_2^2$.\\
$~$ Thus the elementary polynomials corresponding to $X_{1}^2 + X_1  X_2$ are:\\
{ $f_{1_v}=(G_{11}^2 + G_{11}G_{21}),
f_{2_v}=(G_{11}G_{12} + G_{11}G_{22}), f_{3_v}=(G_{12}G_{11} + G_{12}G_{21}), f_{4_v}=(G_{12}^2 +
G_{12}G_{22})$. }\\
\indent Note that for any $v \in V^{\otimes m}$, the elementary polynomials $f_{i_v}$ all have degree $m$.  If
$f_v
\in EP_v$ is an elementary polynomial and $g=g_{ij} \in G(\bar K)$ is any element,
then by $f_v(g)$, we mean the element of $\bar K$ obtained by
substituting $G_{ij}=g_{ij}$ in $f_v$. 
  
Let $v \in V^{\otimes m}$ (or equivalently in $R^m$) be a non-semistable vector
for the action of $SL(V)$. Let $v=\sum a_i \cdot w_i$ be the expansion of $v$ in terms of the basis
vectors. Let $\lambda(t)=\lambda_{ij}(t)$ be a 1-PS subgroup of $G({\bar K})$ which is an instability 1-PS
for $v$. Then there exists an element $g=(g_{ij}) \in G({\bar K})$ such that $g \cdot \lambda(t) \cdot g^{-1}$
is of the form 

\[
\begin{pmatrix}
t^{a_1} & 0 & \cdots & 0 \\
0 & t^{a_2} & \cdots & 0 \\
\vdots & \vdots & \ddots & \vdots \\
0 & 0 & \cdots & t^{a_n}
\end{pmatrix}
\]

for some $a_1,...,a_n$ such that $a_1 \geq a_2 \geq ... \geq a_n$.\\ 
  Then $g \lambda(t) g^{-1} = \lambda'(t)$ is a instability 1-PS for $g \cdot v$ with
$\nu(\lambda,v)=\nu(\lambda',gv)$. Let $g.v=\sum b_iw_i$. Clearly $b_i=f_{i_\nu}(g_{ij})$. Let
$f_{{i_1}_{v}},...,f_{i_{r_v}}$ (resp. $f_{i_{{r+1}_{v}}},...,f_{{i_M}_{v}}$) denote the set of
elementary polynomials in $EP_v$ which vanish at $g$ (resp. are nonzero at $g$). By
lemma {\ref{Noether normalization}}, there exists    
an extension field $L$ of $K$ with $[L:K] \leq rm$ and an $L$-valued point $g' \in G(L)$ such that
$f_{{i_1}_{v}},...,f_{i_{r_v}}$ all vanish at $g'$ and $f_{i_{{r+1}_{v}}},...,f_{{i_M}_{v}}$ are all
non-vanishing at $g'$.
Thus $gv$ and $g'v$ have the same set of monomials with non-zero coefficients. Note that since $\lambda'(t)$
is
of the form 
\[
\begin{pmatrix}
t^{a_1} & 0 & \cdots & 0 \\
0 & t^{a_2} & \cdots & 0 \\
\vdots & \vdots & \ddots & \vdots \\
0 & 0 & \cdots & t^{a_n}
\end{pmatrix}
\]
 an simple observation shows that $m(\lambda',gv)=m(\lambda',g'v)$ and hence
$\nu(\lambda',gv)=\nu(\lambda',g'v)$. Also $\lambda'(t)$ is an instability 1-PS for $g'v$. This is seen as
follows: $g' \lambda g'^{-1}$ is an instability 1-PS of $g'v$ and    
$\nu(\lambda, v)=\nu(g' \lambda g'^{-1}, g'v)$. But $\nu(\lambda,v) = \nu(\lambda',gv) = \nu(\lambda',g'v)$. 
Thus $\nu(g' \lambda g'^{-1},g'v)=\nu(\lambda',g'v)$ and
hence
$\lambda'(t)$ is also an instability 1-PS for $g'v$. This implies that $g'^{-1} \lambda' g'$ is an
instability 1-PS of
$v$.
But $g'^{-1} \lambda' g'$ is defined over $L$. This shows that an instability 1-PS and hence the instability
parabolic of $v$ is defined over $L$.  Since $r \leq n^m$, we see that deg $[L:k] \leq mn^m$. Since $K$ can be
assumed to be seperably closed , the only algebraic extensions possible are those obtained by taking $p^l$-th
roots of generators of $K$ for various non-negative integers $l$. Since $p^t > mn^m$,  it is clear that the
instability parabolic for $v$ is defined over $K^{1/{p^t}}$.
This completes the proof of the lemma.
\end{proof}

Notation: For any integers $n$ and $r$, with $r < n$, set the symbol $nC_r$ ($n$ choose $r$) to be
equal to $n!/(r! (n-r)!)$. \\
\indent We use the above lemma to prove the following theorem:

\begin{theorem} \label{Main theorem}
 Let $G=SL(n)$. Let $X$ be a smooth projective variety over an algebraically closed field $k$. Let $K(X)$
denote its function field. Let $V$, $m$ and $\rho$ be as in lemma {\ref{quantification}}. Let $E$ be any
principal $G$-bundle on $X$. Let $N= \underset{0 \leq r \leq n^m-1}{\text {max}} n^mC_r \cdot (rm)$.  Let
$t$ be any integer such that $p^t > N$.
Let suppose $E$ together together with its first $t$ frobenius pullbacks is semistable. Then
the induced $Sl(V)$ bundle is also semistable. 
\end{theorem}

\begin{proof}
 Let $W=V^{\otimes m}$. Let $E_{Sl(W)}$ denote the induced $Sl(W)$ bundle. We want to show that $E_{Sl(W)}$
is  also semistable. By lemma {\ref{HC}}, this is equivalent to showing that for any maximal parabolic $P$ in
$Sl(W)$ and any
reduction of structure group to $P$, the instability parabolic for the point $\sigma_\circ$ in
$E(Sl(W)/P)_\circ$
corresponding to this reduction is rational.  Let $E(G)_\circ$ be as before. $E(G)$ acts on
$E(Sl(W)/P)$
which is linearized by the very ample line bundle $E(\mathcal L)$ explained before. Since $E_\circ$ gets
trivialized after a finite
seperable extension, we get isomorphisms $E_\circ \otimes_{k(X)} k(X)_s \simeq G \otimes_{k(X)} k(X)_s$ and
$E(Sl(W)/P)_\circ
\otimes_{k(X)} k(X)_s \simeq (SL(W)/P) \otimes_{k(X)} k(X)_s$, the isomorphisms being compatible with the
action. Since $P$ is
a maximal parabolic,
$Sl(W)/P$ is isomorphic to the grassmannian of $r$ dimensional subspaces of $W$ for some $r < $ dim $W$. Using
$E(\mathcal L)_\circ \otimes_{k(X)} k(X)_s$, we get an $G(k(X)_s)$-equivariant embedding of $E(Sl(W)/P)_\circ
\otimes_{k(X)} k(X)_s$ inside $\mathbb{P}(\wedge^r(W))$. We need to show that for this action of
$G(k(X)_s)$ on $\mathbb{P}(\wedge^r(W))$, the instability parabolic for the point $\sigma_\circ$ corresponding
to
this reduction is rational. By lifting this point to a point in $\wedge^r(W)$ (call it $\sigma_\circ$ again),
it boils down to proving the same fact for the action of $G(k(X)_s)$ on $\wedge^r(W)$. This representation of
$G$ on $\wedge^{r}(W)$  is the standard representation of $G$ on $\wedge^r(W)$, induced from the tensor
power representation of $G$ on $V^{\otimes m}$. 
     
Corresponding to the basic $X_1,...,X_n$ of $V$, we get a standard basis of $\wedge^r(W)$ consisting of
vectors of the form $w_{i_1} \wedge..\wedge
w_{i,r}, {(i_1,...,i_r) \in {1,...,M}}~\text{with}~i_1 < ... < i_r$, where each $w_i$ is a noncommutative
monomial in the
$X_i$'s of degree $m$ as in lemma {\ref{quantification}}. Choose an ordering of this basis. Let $\{
W_1,\cdots, W_S \}$ denote
the ordered basis. Note that
$S=n^mC_r$. For any non-semistable vector $v \in \wedge^r(V^{\otimes m})$, let $v = \sum b_iW_i$ be its
expansion in
terms of the basis vector $W_i$. Define the elementary polynomials $EP_v$ similar to lemma
{\ref{quantification}} to be the polynomials in $K(X)[G_{ij}]$ occuring in the coefficients of the image of
$W_i$
when acted upon by a $n \times n$ matrix of indeterminates $G_{ij}$.\\
Note that the degree of the elementary polynomials in now $mr$. Now by using
the same argument as in lemma {\ref{quantification}} by considering $W_i$'s instead of the
noncommutative monomials $w_i$'s in lemma {\ref{quantification}}, we see that the instability parabolic for
the vector
$\sigma_\circ$ is defined over an extension field $[L:k(X)]$, where deg $[L:k(X)] \leq n^mC_r(rm) < p^t$. 
Hence
by lemma \ref{quantification}, for any reduction of structure group to any maximal
parabolic $P$ in $Sl(W)$ the instability 1-PS and hence the instability
parabolic corresponding to $\sigma_\circ$ is defined over $k(X)^{1/p^t}$. Now
consider the action of $F^{t^*}(E(G))_\circ$ on
$F^{t^*}(E(Sl(W)/P))_\circ$. For this action $F^{t^*}(\sigma_\circ)$ has its instability 1-PS and hence
its instability parabolic defined over $k(X)$ via the isomorphism in the commutative diagram shown
below:

  \[
\xymatrix{
        \Spec K \ar[r] \ar[d]  & \Spec K   \\
        \Spec K^{p^{-t}} \ar[ur]_{\simeq} & }
\]

Let $F^{t^*}(\pi):F^{t^*}(E(Sl(W)/P) \rightarrow X$, denote the pullback of $\pi$ under $F^{t}$. Similarly
let  $F^{t^*}(T_\pi)$ denote the pullback of the relative tangent bundle of $E(Sl(W)/P)$ under $F^t$ which is
the same as the relative tangent bundle of the pullback of $E(Sl(W)/P)$ under $F^t$. Since $F^{t^*}(E)$ is
semistable and the instability 1-PS corresponding to every reduction to every maximal parabolic is rational,
deg $F^{t^*}(\sigma)^* F^{t^*}(T_\pi)> 0$. But deg $F^{t^*}(\sigma)^*(F^{t^*}(\pi))=$ deg
$F^{t^*}(T_\sigma)= p^t \cdot$deg
$T_\sigma$. This follows from the fact that for any line bundle $L$ on $X$, $F^{t^*}(L)$ is isomorphic to 
$L^{p^t}$. Hence deg $T_\sigma > 0$ for every reduction of $E_{Sl(W)}$ to every maximal parabolic $P$ in
$Sl(W)$ and hence by lemma {\ref{HC}}, $E_{Sl(W)}$ is also semistable.  
\end{proof}

Now let $\rho'$ be an arbitrary representation of $Sl(V)$. We use the above lemma to get bounds for the number
of semistable frobenius pullbacks required for an $SL(V)$-bundle $E$, so that the induced bundle on extension
of structure group via $\rho'$ is again semistable.

 Let $\rho':Sl(V) \rightarrow Sl(W)$ be an arbitrary representation of $Sl(V)$. Let $0=W_0 \subset W_1 \subset
\cdots \subset W_l = W$ be the Jordan-Holder filtration of $W$. Then $W_i/W_{i-1}$ are simple $Sl(V)$-modules.
Any
simple $Sl(V)$-module $L(\lambda)$ corresponding to a highest weight vector $\lambda=\sum a_i\omega_i$ is an
$Sl(V)$-submodule of $V^{\otimes \mid \lambda \mid}$, where $\mid \lambda \mid = \sum ia_i$ is called the
degree of $\lambda$. Following Langer (see [L]), we call the maximum of the degrees of the dominant weights
whose modules occur as the successive
quotients in the Jordan-Holder filtration as the Jordan-Holder degree of $W$, denoted JH($W$).

\begin{lemma} \label{Arb rep of Sl(V)}
 Let $\rho': Sl(V) \rightarrow Sl(W)$ be an arbitrary representation of $W$. Let JH($W$)= $d$. Let $E$ be a
$Sl(V)$-bundle on $X$ such that $F^{t^*}(E)$ is semistable for some $t$ such that $p^t > \underset{0 < r \leq
n^d-1}{\text
{max}} n^dC_r(rd)$.
Then the induced $Sl(W)$-bundle is also semistable.
\end{lemma}

\begin{proof}
 Let  $0=W_0 \subset W_1 \subset \cdots \subset W_l = W$ be the Jordan-Holder filtration of $W$ as before.
Then each
successive quotient is a $Sl(V)$-submodule of the $Sl(V)$-module $V^{\otimes i}$, for some $i \leq
d$. Since $F^{t^*}(E)$ is semistable, by lemma \ref{Main theorem} each of the induced vector bundles
obtained by extension of structure group of $E$ using these tensor power representations are also semistable
and of degree zero. Since a degree zero subbundle of a semistable bundle of degree zero is also semistable ,
we see that the induced vector bundle $E_{Sl(W)}$ is filtered by semistable bundles of degree zero and
hence $E_{Sl(W)}$ is also semistable of degree zero. This completes the proof of the lemma.
 
\end{proof}

\begin{remark}

Let $V$ be a vector space defined over $K$. Let $0 \rightarrow W'
\rightarrow W \rightarrow W'' \rightarrow 0$ be
a short-exact sequence of $Sl(V)$-modules defined over $K$. If the instability parabolic for each of the
unstable $K$-rational points in $W$ is defined over $K^{1/p^t}$ then so is the case for all the unstable
$K$-rational points in $W'$. However it does not seem easy to bound the field of definition of the instability
parabolics for unstable $K$-rational points in $W$ in terms of similar bounds for $W'$ and $W''$. Similarly
it does not seem possible to determine the field of definition of
the instability parabolic for all the unstable $K$-rational points in $W''$ knowing the same for $W$.
This is because unstable $K$-rational points in $W$ may not surject onto the unstable $K$-rational points in
$W''$. However if an integer $t$ satisfies the property that any $Sl(V)$-bundle with first $t$-frobenius
pullbacks semistable induces semistable $Sl(W')$ and $Sl(W'')$-bundles on extensions of structure
group, then clearly the induced $Sl(W)$-bundle is also
semistable. Similarly, if integer $s$ satisfies the property that any $Sl(V)$-bundle with first $s$-frobenius
pullbacks semistable induces a semistable $Sl(W)$ on extension of structure
group, then clearly the induced $Sl(W'')$-bundle is also
semistable. This is because any degree zero quotient of a semistable bundle of
degree zero is also semistable of degree zero. Hence for computing the number of semistable frobenius
pullbacks required for a $Sl(V)$-bundle to induce a semistable bundle on extension of structure group, it
suffices to compute the same for the tensor power representation. Then using the fact that an arbitrary
representation $W$ of $SL(V)$ can be filtered by $Sl(V)$-modules which are submodules of a suitable
tensor-power representation, we get bounds for the number of semistable frobenius pullbacks required so that
the induced $Sl(W)$-bundle is semistable.
 
\end{remark}

\begin{remark}
 Note that one of the major differences between the methods for estimating the field of
definition of the instability parabolic described here and the methods of [RR] and [CH] is that unlike their
methods we do not use the orbit map $E(G)_\circ \times E(G/P)_\circ \rightarrow
E(G/P)_\circ$ and try and bound its non-seperability. We directly estimate the field of definition of the
instability parabolic which is probably
weaker than trying to bound the non-reducedness of the stabilizers of the various unstable $K$-rational points
which does not seen quantifiable. Indeed it is an open problem as to whether it is possible to have a
representation of a semisimple group $G$ such that the
stabilizers of some of the unstable $K$-rational points in the representing space are non-reduced but
their instability parabolics are rational. We do not know the answer to this yet.
\end{remark}


\section{Case of an arbitrary reductive group} \label{Case of a arbitrary reductive group}
 In this section we get bounds for the field of definition of the instability parabolic for an arbitrary
representation of an arbitrary reductive algebraic group.\\
$~$ Let $G$ be a reductive algebraic group defined over $k$. Fix an embedding $i: G \hookrightarrow Gl(V)$,
where $V$ is a $n$-dimensional vector space. Fix a maximal torus $T$ in $G$. 

\begin{theorem}\label{quantification in general}
 
 Let $[F:k]$ be an extension of fields. Let $\rho: G \rightarrow Gl(W)$ be a finite dimensional
representation of $G$ defined over $F$. Then there
exists an integer $t$, such that for any unstable $F$-rational point in $W$, its instability parabolic is
defined over $F^{1/p^t}$.

\end{theorem}

\begin{proof}
   The proof of the lemma is similar to the proof of lemma \ref{quantification}. The main difference now is
that we also have to consider the defining equations of $G$ in $Gl(V)$ along with the elementary polynomials
of the unstable $F$-rational points. As before we may assume that $F$ is seperably closed. Let dim $W=
m$ and dim $V$= $n$. Fix a basis of $V$ via which $Gl(V)$
will be identified with $Gl(n,F)$. $Gl(n,F)$ will be thought of as an open subscheme of
$M_n(F)$ which will identified with $\mathbb{A}^{n^2}_F$.
Let the affine coordinate ring of $G$ for the embedding $\tilde i: G \hookrightarrow M_n(F)$ given by the
composite of $G \hookrightarrow Gl(V) \underset {\text {open subscheme}}{\subset} M_n(F) \simeq 
\mathbb{A}^{n^2}_F$ = Spec $F[G_{ij}]_{1\leq i,j \leq n}$
be isomorphic to $F[G_{ij}, (\text{det}~G_{ij})^{-1}]/(h_1,...,h_s)_{1\leq i,j \leq n}$, for some $h_1,...,h_s
\in F[G_{ij}, (\text{det}~G_{ij})^{-1}]$. 
The valued points of $Gl(V)$ will be thought of
as $n \times n$ invertible matrices. The affine coordinte ring of $Gl(V,F)$ is isomorphic to $A=
F[G_{11}, \cdots ,G_{nn},$ (det $G)^{-1}]$, where det$(G)$ is the determinant polynomials in the $G_{ij}$'s
and a matrix element $g=g_{ij} \in Gl(V,L)$, for any extension field $[L:F]$, will be thought of as an
$L$-valued
point of Spec $A$ in the obvious way. Choose an ordered
simultaneous eigen basis $\{w_1, \cdots ,w_m\}$ of $W$ for all the 1-PS of $G$ which
lie in $T$. With respect to this basis, the matrix of $\rho$ will be an $m \times m$ matrix whoses
entries are regular functions on $G$, which are by definition the restrictions of the regular functions on $
 \mathbb{A}^{n^2}_F$ via the embedding ${\tilde i}$.

  \[
\begin{pmatrix}
\bar{f_{11}}(G_{ij})/ (\overline{\det} (G_{ij}))^{a_{11}} & \bar {f_{12}}(G_{ij})/ (\overline{\det}
(G_{ij}))^{a_{12}} & \cdots & (\bar{
f_{1m}}(G_{ij})/ (\overline{\det} (G_{ij}))^{a_{1m}} \\
\bar {f_{21}}(G_{ij})/ (\overline{\det}  (G_{ij}))^{a_{21}} & (\bar {f_{22}}(G_{ij})/ (\overline{\det}
({G_{ij}))^{a_{22}}} & \cdots & (\bar
{f_{2n}}(G_{ij})/ (\overline{\det} ({G_{ij}}))^{a_{2n}} \\
\vdots & \vdots & \ddots & \vdots \\
(\bar {f_{n1}}(G_{ij})/ (\overline{\det} ({G_{ij}}))^{a_{n1}} & \bar {f_{n2}}(G_{ij})/ (\overline{\det}
G_{ij})^{a_{n2}} & \cdots & (\bar
{f_{nn}}(G_{ij})/ (\overline{\det} ({G_{ij}}))^{a_{nn}}
\end{pmatrix}
\]

where $\bar {f_{ij}}$ and $ (\overline{\det} ({G_{ij}}))$ are regular functions on $G$ which are by
definition the restrictions of the regular
functions $f_{ij}(G_{ij})$ and det($G_{ij}$) resp. from $M_n(F)$ to $G$. By multiplying the
numerator and denominator of each
matrix entry by a suitable power of det $G$, we can assume that all the $a_{ij}$'s occuring in the matrix are
all equal to some non-negative integer , say $a$. Let $w \in W$ be a non-semistable $F$-rational point of
$W$. Let $\lambda(t)$ be an instability 1-PS of $w$. Then there exists an element $g\
\in G$ such that $g\lambda(t)g^{-1} = \mu(t) \subset T$. Clearly $\mu(t)$ is an instability 1-PS of
$gw$.

Now as in lemma \ref{quantification}, we will define the elementary polynomials associated to
$w$ to be  certain \textquotedblleft modified\textquotedblright coefficients that occur in the expansion of
$\rho(G) \cdot w$ in terms of the
basis vectors
$\{w_1, \cdots w_m\}$ . These will  
be certain polynomials in the $F[G_{ij}]_{(1\leq i,j\leq n)}$. \\
More precisely, if $w = \underset{i=1}{\overset{m} \sum} b_iw_i$, then
\[
\rho(G)(w) = \underset{i=1}{\overset{m}\sum} \dfrac{\bar f_{1i}(G_{ij})b_i} {(\overline{\det} ~{
(G_{ij})})^a}w_1 +
\underset{i=1}{\overset{m}\sum} \dfrac{\bar f_{2i}(G_{ij})b_i} {(\overline{\det}~{ (G_{ij})})^a}w_2 + \cdots
+ \underset{i=1}{\overset{m}\sum} \dfrac{\bar f_{mi}(G_{ij})b_i} {(\overline{\det}~{ (G_{ij})})^a}w_m.
\]

Define the elementary polynomials of associated to $w$, denoted
EP$_w$, to be the polynomials $\{F_{1_w}= \underset{i=1}{\overset{m}\sum}
f_{1i}(G_{ij})b_i~;~\cdots~;~
F_{m_w}= \underset{i=1}{\overset{m}\sum} f_{mi}(G_{ij})b_i\}$.

Then clearly,
\[gw= \underset{i=1}{\overset{m}\sum} \dfrac{\bar f_{1i}(g_{ij})b_i} {(\overline{\det} ~{ (g_{ij})})^a}w_1 +
\underset{i=1}{\overset{m}\sum} \dfrac{\bar f_{2i}(g_{ij})b_i} {(\overline{\det}~{ (g_{ij})})^a}w_2 + \cdots
\underset{i=1}{\overset{m}\sum} \dfrac{\bar f_{mi}(g_{ij})b_i} {(\overline{\det}~{ (g_{ij})})^a}w_m.
\]
\[= \dfrac{F_{1_w}(g_{ij})w_1 + \cdots + F_{m_w}(g_{ij})w_m}{ (\det~(g_{ij}))^a}\]

Thus we see that the vanishing or non-vanishing of a particular coefficient of $gw$ depends on whether or
not
the corresponding elementary polynomial vanishes at $g$ or not.

Let $F_{{i_1}_w} , \cdots ,F_{i_{r_w}}$ be exactly the set of elementary polynomials which are vanishing
at $g$. Now as in
lemma \ref{quantification}, we would like to find a quantifiable extension $[L:F]$ and a element $g' \in G(L)$
such that $g'w$ has the same set of coefficients as $gw$ which are zero. 
Consider the affine $F$-algebra $B = F[G_{ij}]/(h_1, \cdots , h_s, F_{{i_1}_w},\cdots,F_{i_{r_w}})$. Let $G_w$
be the product of all the elementary polynomials $F_{i_w}$ which are non-vanishing at $g$. Note that
$g=(g_{ij})$ is a $\bar F$-valued point of Spec $A$ at which $G_w$ in non-vanishing. Hence by lemma
\ref{Noether normalization}, there exists an extension field $[L:F]$ with deg $[L:F] \leq {\text {deg}} 
(\underset{j=1}{\overset{r} \prod}F_{{i_j}_w} \cdot \underset{j=1}{\overset{s} \prod}h_j) = d~({\text {say}})$
and an $L$-valued
point $g'$ of Spec $A$ at which $G_w$ is non-vanishing. Since the polynomials
$h_1,..,h_s$ vanish at $g'$, it follows that $g' \in G(L)$. Now
$gw$ and $g'w$
have the same set of coefficients of the $w_i$'s which are non-zero. Hence as in the proof of
lemma \ref{quantification}, $\mu(t)$ is also an 
instability 1-PS of $g'w$ and hence $g'^{-1} \mu(t) g'$ is an instability 1-PS of $w$, which is clearly
defined over $L$. From this it follows easily that if $t$ is any integer such that $p^t > d$, then for any
unstable $F$-rational point $w \in W$ it has an instability 1-PS and hence its instability parabolic defined
over $F^{1/p^t}$.
\end{proof}

\begin{definition}
 Let $G$ be a reductive algebraic group defined over $K$. Let $W$ be a vector space defined over $K$. Let
$\rho: G \rightarrow Gl(W)$ be a representation of $G$ defined over $K$. We say that the {\bf instability
parabolic for subspaces is defined over L} 
if for the induced
action of $G$ on $\wedge^i(W)$, the 
instability parabolic for any unstable $K$-rational point in $\wedge^i(W)$ is defined over $L$ for all $i$
with $0 < i \leq m$. Similarly we will say that the {\bf instability parabolic for subspaces is rational} if
$L$ can be choosen to be $K$.
\end{definition}

\begin{corollary}\label{Corollary to quantification in general}
 Let $G$ and $\rho$ be as in the above definition. Then there exists an integer $t'$ such that the instability
parabolic for
subspaces is defined over $K^{1/p^{t'}}$.  Consequently, if $E$
is any principal $G$-bundle on $X$ such that
$F^{t'^*}(E)$ is semistable then the induced vector bundle $E_W$ is also semistable.

\end{corollary}

\begin{proof}
Proof follows immediately from theorem \ref{quantification in general} and the proof of theorem \ref{Main
theorem}.
\end{proof}


\begin{thebibliography}{1111}
\bibitem[CH]{CH}
Fabrizio Coiai; Yogish, Holla. Extension of structute group of principal bundles in positive characteristic. 
J. reini. agnew. Math. 595 (2006), 1-24

\bibitem[RR]{BH}
S. Ramanan and A. Ramanathan, Some remarks on the instability flag, Tohoku Math. J. 36 (1984), 269-291

\bibitem[Kempf]{BH}
G. Kempf, Instability in invariant theory, Ann. of Math. 108 (1978), 299-316.

\bibitem[L]{BH}
Langer, Adrian, Semistable principal $G$-bundles in positive characteristic.  Duke Math. J.  128  (2005), 
no. 3, 511-540.
\end{thebibliography}
\end{document}